\documentclass{amsart}
\usepackage{mathtools, microtype, mathrsfs, xfrac, hyperref, amssymb, amsthm, enumerate, xspace, stmaryrd}
\usepackage{thmtools}
\usepackage{cleveref}
\usepackage[latin1]{inputenc}
\usepackage{tikz-cd}
\usepackage{todonotes}

\numberwithin{equation}{section}

\numberwithin{subsection}{section}

\allowdisplaybreaks[1]

\newenvironment{enumeratea}
{\begin{enumerate}[\upshape (a)]}
{\end{enumerate}}

\newenvironment{enumerate1}
{\begin{enumerate}[\upshape (1)]}
{\end{enumerate}}

\newtheorem{theorem}{Theorem}[section]
\newtheorem{proposition}[theorem]{Proposition}
\newtheorem{proposition-definition}[theorem]
{Proposition-Definition}
\newtheorem{corollary}[theorem]{Corollary}
\newtheorem{lemma}[theorem]{Lemma}

\theoremstyle{definition}
\newtheorem{definition}[theorem]{Definition}

\newtheorem{example}[theorem]{Example}

\theoremstyle{remark}

\renewcommand{\mathcal}{\mathscr}

 \newcommand\cB{\mathcal{B}}

\newcommand\cG{\mathcal{G}} \newcommand\cH{\mathcal{H}}
 
 \newcommand\cL{\mathcal{L}}
\newcommand\cM{\mathcal{M}} 
\newcommand\cO{\mathcal{O}} 
 
\newcommand\cS{\mathcal{S}} 
 \newcommand\cV{\mathcal{V}}
\newcommand\cW{\mathcal{W}}

\renewcommand\AA{\mathbb{A}} 
 
 \newcommand\FF{\mathbb{F}}
\newcommand\GG{\mathbb{G}}

 \newcommand\PP{\mathbb{P}}

 \newcommand\ZZ{\mathbb{Z}}

\newcommand\rI{\mathrm{I}}

 \newcommand\rR{\mathrm{R}}
\newcommand\rS{\mathrm{S}} \newcommand\rT{\mathrm{T}}

\newcommand\rmm{\mathrm{m}}

 \newcommand\frm{\mathfrak{m}}

\newcommand\arr{\ifinner\to\else\longrightarrow\fi}

\newcommand\arrto{\ifinner\mapsto\else\longmapsto\fi}

\newcommand{\xarr}{\xrightarrow}

\renewcommand\H{\operatorname{H}}

\newcommand{\eqdef}{\mathrel{\smash{\overset{\mathrm{\scriptscriptstyle def}} =}}}

\def\displaytimes_#1{\mathrel{\mathop{\times}\limits_{#1}}}

\def\displayotimes_#1{\mathrel{\mathop{\bigotimes}\limits_{#1}}}

\renewcommand\hom{\operatorname{Hom}}

\newcommand\aut{\operatorname{Aut}}

\newcommand\spec{\operatorname{Spec}}

\newcommand\indlim{\varinjlim}

\newcommand{\underisom}{\mathop{\underline{\mathrm{Isom}}}\nolimits}

\newcommand{\underaut}{\mathop{\underline{\mathrm{Aut}}}\nolimits}

\newlength{\ignora}

\newcommand{\mmu}{\boldsymbol{\mu}}

\newcommand{\gm}{\GG_{\rmm}}

\newcommand{\GL}{\mathrm{GL}}

\DeclareFontFamily{U}{mathx}{\hyphenchar\font45}
\DeclareFontShape{U}{mathx}{m}{n}{
      <5> <6> <7> <8> <9> <10>
      <10.95> <12> <14.4> <17.28> <20.74> <24.88>
      mathx10
      }{}
\DeclareSymbolFont{mathx}{U}{mathx}{m}{n}
\DeclareFontSubstitution{U}{mathx}{m}{n}
\DeclareMathAccent{\widecheck}{0}{mathx}{"71}
\DeclareMathAccent{\wideparen}{0}{mathx}{"75}

\renewcommand{\epsilon}{\varepsilon}

\newcommand{\cha}{\operatorname{char}}

\newcommand{\aff}[1][k]{(\mathrm{Aff}/#1)}

\begin{document}

\title[Neutral representations and fields of moduli]{Neutral representations\\of finite diagonalizable group schemes\\and fields of moduli}

\author{Giulio Bresciani}
\author{Angelo Vistoli}
\author{Tianzhi Yang}

\email[Bresciani]{giulio.bresciani1@unipi.it}
\email[Vistoli]{angelo.vistoli@sns.it}
\email[Yang]{tianzhi.yang@sns.it}

\address[Bresciani]{Dipartimento di Matematica, Università di Pisa, Italy}
\address[Vistoli, Yang]{Scuola Normale Superiore, Pisa, Italy}

\thanks{The first and second authors were partially supported by research funds from Scuola Normale Superiore, project \texttt{SNS19\_B\_VISTOLI}, and by PRIN project ``Derived and underived algebraic stacks and applications''.}

\maketitle

\begin{abstract}
We introduce the notion of a \emph{neutral representation} of a finite group, or finite group scheme, $G$; a representation $V$ with the property that if a gerbe $\cG$ over a field $k$ that is a form of the classifying stack $\cB G$  admits a vector bundle that is a form of $V$, then it is neutral, that is, $\cG(k)$ is not empty. We give some criteria for a representation of a finite diagonalizable group scheme to be neutral. 

We apply this notion to give wide classes of examples of smooth curves, or varieties with a marked point, with cyclic automorphism groups, which are defined over their field of moduli, greatly generalizing some of the results in \cite{giulio-angelo-arithmetic}.
\end{abstract}

\section{Introduction} 

The motivation for this paper comes from the theory of fields of moduli. The concept of \emph{field of moduli} was introduced by Matsusaka in \cite{matsusaka-field-of-moduli}, and considerably clarified by Shimura in \cite{shimura-automorphic}. In the paper \cite{giulio-angelo-arithmetic}, inspired by \cite{debes-emsalem}, this concept is generalized and put in the general context of the theory of gerbes, in the sense of Grothendieck and Giraud \cite{giraud}. We refer to the introduction of \cite{giulio-angelo-arithmetic} for a general discussion of the history of the problem.

Let $k$ be a field, $\overline{k}$ its algebraic closure. Let $X$ be a proper variety, possibly with an additional structure, for example a polarization, or a set of marked points, with finite automorphism group over $\overline{k}$. Under some fairly weak hypotheses there exist an intermediate extension $k \subseteq \ell \subseteq \overline{k}$, finite over $k$, and a finite gerbe $\cG_{X} \arr \spec \ell$ of twisted forms of $X$, called the \emph{residual gerbe} of $X$; the field $\ell$ is the field of moduli of $X$.

The basic question is: is $X$ defined over its field of moduli? In other words, is $\cG_{X}(\ell) \neq \emptyset$, or is the gerbe $\cG_{X}$ \emph{neutral}? In \cite{giulio-angelo-arithmetic} we give a non-trivial answer, based on the concept of \emph{$R$-singularity}. If $X$ is a $d$-dimensional variety over $\overline{k}$ with a smooth marked point $x_{0} \in X$, and its group of automorphism has the $R_{d}$-property \cite[Definition~6.12]{giulio-angelo-arithmetic}, then the pair $(X, x_{0})$ is defined over its field of moduli. This generalizes several results in the literature.

In the present paper we introduce a new idea, which allows us to give many new examples of classes of varieties that are defined over their field of moduli. If $V$ is a cohomology group intrinsically associated with $X$, for example $\H^{i}(X, \cO_{X})$ for some $i \geq 0$, this becomes a representation of the group $G$ of automorphisms of $X$. Also, $V$ descends to a vector bundle $\cV \arr \cG_{X}$ over the residual gerbe. We show that there is an interesting class of examples of representations $V$ of finite groups $G$, or, more generally, of finite group schemes, with the property that if the representation descends to a vector bundle $\cV \arr \cG$ on a form $\cG$ of $\cB G$, then $\cG$ is neutral. We call these \emph{neutral representations}: see \Cref{neutral} for the exact definition.

This paper treats the case of representations of diagonalizable groups; in particular, we get results about neutral representations of finite abelian groups of order prime with the characteristic of the base field. For simplicity, this is the only case that we will consider in the introduction.

To give a flavor of our results, in this introduction we only give the simplest example of neutral representation, which is however enough to show that the theory is nontrivial.

\begin{theorem}
Let $p$ be a prime, $G$ a cyclic group of order~$p$, $V$ a finite-dimensional representation of $G$ over an algebraically closed field of characteristic different from $p$. If $\dim V - \dim V^{G}$ is not divisible by $p$, then the representation is neutral.
\end{theorem}

This is a particular case of \Cref{easy-cyclic}. The case $p = 2$ is already contained in \cite{giulio-angelo-arithmetic}.

Here are two corollaries, which cannot be proved with the techniques of \cite{giulio-angelo-arithmetic}.

\begin{corollary}
Let $X$ be a smooth projective curve over $\overline{k}$, $p$ a prime different from $\cha k$. Assume the following two conditions.

\begin{enumerate1}

\item The automorphism group $\aut X$ of $X$ is cyclic of order~$p$.

\item The difference between the genus of $X$ and the genus of $X/\aut X$ is not divisible by $p$.

\end{enumerate1}

Then $X$ is defined over its field of moduli.
\end{corollary}

This is a particular case of \Cref{curve-cyclic}. The case $p = 2$ and $X/\aut X = \PP^{1}$ follows from standard results on hyperelliptic curves.

\begin{corollary}
Let $(X, x_{0})$ be an integral variety over $\overline{k}$ with a smooth rational point $x_{0} \in X(k)$, $p$ a prime different from $\cha k$. Assume the following two conditions.

\begin{enumerate1}

\item The automorphism group $\aut (X, x_{0})$ is cyclic of order~$p$.

\item The difference between $\dim X$ and the dimension of the fixed point locus $X^{G}$ at $x_{0}$ is not divisible by $p$.

\end{enumerate1}

Then $(X, x_{0})$ is defined over its field of moduli.
\end{corollary}

This is a particular case of \Cref{marked-cyclic}. The case in which $p = 2$ and $x_{0}$ is an isolated fixed point is the content of \cite[Corollary~6.25]{giulio-angelo-arithmetic}.

In \cite{giulio-singularities}, the first author classified $R$-singularities in dimension $2$: as we shall explain in \Cref{R<->neutral}, this is equivalent to classifying faithful, neutral representations in dimension $2$ without pseudoreflexions. In the forthcoming paper \cite{giulio-tianzhi}, the first and third author complete the classification of faithful, neutral representations in dimension $\le 3$ and characteristic $0$. Furthermore, they introduce new methods for studying faithful, neutral representations of abelian groups in arbitrary dimensions.

\subsection*{Contents of the paper} \Cref{general-setup} contains the definition of a neutral representation. 

In \Cref{R-singularity} we explain the close connection between neutral representations and $R$-singularities, a central tool in \cite{giulio-angelo-arithmetic}. Our methods allow us to give many more, but not all, examples of $R$-singularities.

\Cref{blended} introduces the basic tool that we use, the \emph{blended decomposition} of a representation. 

\Cref{periods} contains an auxiliary technical result that is undoubtedly known to experts, but for which we do not have a reference.

The main results are stated in \Cref{main-theorems}. Together, they give a general condition for a representation to be neutral. The statements are somewhat involved, but they become much simpler for cyclic groups: see \Cref{cyclic-general} for the more general statement, and \Cref{easy-cyclic} for a particularly simple statement. The proofs of the main theorems are in \Cref{proofs}.

Finally, \Cref{applications} contains our applications to fields of moduli of curves and pointed varieties with cyclic automorphism groups, \Cref{curve-cyclic} and \Cref{marked-cyclic}.

\section{General setup, and neutral representations}\label{general-setup}

If $K$ is a field, we denote by $\aff[K]$ the category of affine schemes over $K$. If $G$ is a finite group scheme over a field $K$, or a finite group, we denote by $\cB_{K}G \arr \aff[K]$ the classifying stack of $G$-torsors over $K$-algebras.

All the gerbes we will consider are going to be finite gerbes, in the sense of \cite[Definition~4.1]{borne-vistoli}, over fields; this means the following. Let $K$ be a field, $\cG \arr \aff[K]$ a category fibered in groupoids. We say that $\cG$ \emph{a finite gerbe over $K$} (or, from now on, simply \emph{a gerbe}) if the following conditions are satisfied.

\begin{enumerate1}

\item $\cG \arr \aff[K]$ is a gerbe in the fppf topology.

\item $\cG$ is limit-preserving, in the sense that if $\{A_{i}\}_{i\in I}$ is a filtered system of $K$-algebras, the natural functor $\indlim_{i}\cG(A_{i}) \arr \cG(\indlim_{i}A_{i})$ is an equivalence of categories.

\item The diagonal $\cG \arr \cG \times\cG$ is representable and finite.

\end{enumerate1}

A gerbe $\cG \arr \aff[K]$ is \emph{neutral} if $\cG(K) \neq \emptyset$. Such a gerbe should properly be called ``neutralizable'', and the term ``neutral'' should be reserved for a gerbe with a specified section in $\cG(K)$, but the distinction is not important for our purposes, and we will go with the simpler word. If $\xi \in \cG(K)$ and $G \eqdef \underaut_{K}\xi$, then $\cG$ is equivalent to $\cB_{K}G$.

The conditions above imply that there exists a finite extension $L/K$, a finite group scheme  $G \arr \spec L$, and an equivalence $\cG_{L} \simeq \cB_{L}G$. A finite gerbe over $K$ is a smooth Artin stack of finite type over $K$.

Two gerbes, $\cG$ over $K$ and $\cH$ over $L$, are \emph{equivalent} if there exists a common extension $K \subseteq E$ and $L \subseteq E$ such that $\cG_{E}$ and $\cH_{E}$ are equivalent as fibered categories over $\aff[E]$.

\subsection*{Neutral representations}

Let $K$ a field, $G$ a finite group scheme $K$, $G \arr \GL_{K}(V)$ a finite-dimensional representation of $G$. We can view $V$ as a vector bundle $[V/G] \arr \cB_{K}G$.

If $\cG \arr \spec k$ is a gerbe and $\cV \arr \cG$ is a vector bundle, we say that \emph{the pair $(\cG, \cV)$ is equivalent to $(\cB_{K}G, V)$}, or that $(\cG, \cV)$ is \emph{a model of $(\cB_{K}G, V)$}, if there exists a common extension $E$ of $k$ and $K$, and an equivalence $\alpha\colon \cB_{E}G_{E} \simeq \cG_{E}$, with an isomorphism of vector bundles $\phi\colon [V/G]_{E} \simeq \alpha^{*}\cV_{E}$. over $\cB_{E}G$.

Such a pair $(\alpha, \phi)$ will be called \emph{an equivalence of $(\cB_{K}G, V)$ with $(\cG, \cV)$}. Alternatively, a vector bundle $\cV \arr \cG$ on a gerbe $\cG$ such that $(\cG, \cV)$ is equivalent to $(\cB_{K}G, V)$ will be called \emph{a $V$-form}.

If $G$ is a finite group and $G_{K}$ is the corresponding discrete group scheme over $K$, representations $G \arr \GL_{K}(V)$ correspond to representations $G_{K} \arr \GL_{K}(V)$; from now on we will identify representations of $G$ and representations of $G_{K}$.

The following is the main definition of our paper. 

\begin{definition}\label{neutral}
Let $K$ be a field, $G$ a finite group scheme of over $K$, $G \arr \GL_{K}(V)$ a finite-dimensional representation of $G$. We say that $V$ is \emph{neutral} if, for every $V$-form $\cV \arr \cG$, the gerbe $\cG$ is neutral.
\end{definition}

Of course, this definition could be greatly generalized; one could also consider non-finite group schemes, and sets of representations rather than a single representation. One can also fix a base field $k$, and only consider gerbes over extensions of $k$. These generalizations will not be pursued in the present paper.


If $K'$ is an extension of $K$, and $V$ is a representation of $G$ over $K$, then the representation $V_{K'}$ over $K'$ by extension of scalars is neutral if and only if $V$ is neutral. So, in discussing what representations are neutral we can always assume that $K$ is algebraically closed.

\section{Neutral representations and $R$-singularities}\label{R-singularity}

Let us recall the notion of $R$-singularity. The following definitions are contained in \cite{giulio-angelo-arithmetic}.

\begin{definition}\hfil
\begin{enumeratea}

\item A \emph{singularity} over a field $k$ is a pair $(S,s)$, where $S$ is a scheme of finite type over $k$, and $s \in S(k)$ is a rational point.

\item Two singularities $(S,s)$ and $(S', s')$ over $k$ are \emph{equivalent} if $\widehat{\cO}_{S,s}$ is isomorphic to $\widehat{\cO}_{S',s'}$ as a $k$-algebra. By Artin's approximation theorem, this is equivalent to saying that $(S,s)$ and $(S's')$ have isomorphic étale neighborhoods.

\item Two singularities $(S,s)$ and $(S', s')$ over two fields $k$ and $k'$ with the same characteristic are \emph{stably equivalent} if there exists a common extension $k \subseteq k''$ and $k' \subseteq k''$ such that $(S,s)_{k''}$ and $(S',s')_{k''}$ are equivalent.

\item A singularity $(S,s)$ is a \emph{tame quotient singularity} if it stably equivalent to a singularity of type $(V/G, [0])$, where $G$ is a finite discrete group with $\cha k \nmid |G|$ and $V$ is a representation of $G$ over an algebraically closed field $K$. By Cartan's theorem, this is equivalent to the existence of a variety $V$ over an algebraically closed field $K$ with a smooth point $v_{0} \in V(K)$, together with the action of a finite discrete group with $\cha k \nmid |G|$ on $V$ leaving $v_{0}$ fixed, such that $(S,s)$ is stably equivalent to $(V/G, [v_{0}])$. 

\item A singularity $(S,s)$ over $k$ is \emph{liftable} if given a resolution of singularities $\widetilde{S}\arr S$, there is a $k$-rational point of $\widetilde{S}$ lying over $s$. By the Lang--Nishimura's theorem, this condition does not depend on the resolution $\widetilde{S}$.

\item An \emph{$R$-singularity} is a tame quotient singularity $(S,s)$, such that every singularity $(S',s')$ that is stably equivalent to $(S,s)$ is liftable.

\end{enumeratea}
\end{definition}

$R$-singularities and neutral representations are strictly linked.

\begin{theorem}\label{R<->neutral}
Let $G$ be a finite group with $p \nmid |G|$, and let $G \subseteq \GL_{K}(V)$ be a faithful representation, where $K$ is an algebraically closed field with $\cha K \nmid |G|$. 

\begin{enumerate1}

\item Assume that $(V/G, [0])$ is an $R$-singularity. Then the representation $V$ is neutral.

\item Assume that $G$ does not contain any pseudoreflexion, and that the representation $V$ is neutral. Then $(V/G, [0])$ is an $R$-singularity.

\end{enumerate1}
\end{theorem}

\begin{proof}

Let $\cV \arr \cG$ be a $V$-form over a field $k$ By extending $K$ we may assume that $k \subseteq K$, and that $(\cG, \cV)_{K}$ is equivalent to $(\cB_{K}G, [V/G])$. Call $M$ the moduli space of $\cV$, fix $m_{0} \in M(k)$, and let $\widetilde{M} \arr M$ be a resolution of singularities. The base change $M_{K}$ is isomorphic to $V/G$, and the image $m_{0} \in M(k)$ corresponds to $[O] \in V/G$.

Consider the embedding $\cG \subseteq \cV$ given by the $0$-section; the image of $\cG$ in $M$ is $m_{0}$. The reduced inverse image of $m_{0}$ in $\cV$ is $\cG$. Since $(V/G, [0])$ is an $R$-singularity, the resolution $\widetilde{M}$ has a rational point over $m_{0}$. Then \cite[Corollary~4.4]{giulio-angelo-valuative} shows that $\cG(k)$ is not empty, and the first statement follows.

Now, assume that $G \subseteq \GL_{K}(V)$ does not contain any pseudoreflexion, and $V$ is neutral. Let $(S,s)$ be a singularity over a field $k$ which is stably equivalent to $(V/G, [0])$, and let $\widetilde{S} \arr S$ be a resolution of singularities. Call $\cS \arr S$ the minimal stack of $S$ (see \cite[\S 6.1]{giulio-angelo-arithmetic}). We need to show that $\widetilde{S}$ has a rational point over $k$. Denote by $\cG$ the reduced inverse image of $s$ in $\cS$, and let $\cV \arr \cG$ be the normal bundle of $\cG$ in $\cS$. Again by \cite[Corollary~4.4]{giulio-angelo-valuative} it is enough to show that $\cG$ has a rational point. Since $[V/G] \arr V/G$ is the minimal stack, because $G$ does not contain any pseudoreflexions, and formation of the minimal stack commutes with base change on the base field, we have that $\cS_{K}$ is equivalent to $[V/G]$, and $\cG$ is the substack $[0/G]\subseteq [V/G]$. Hence $\cV \arr \cG$ is a $V$-form, and the conclusion follows.
\end{proof}

\begin{example}
	Let us show that the condition that $G$ does not contain any pseudoreflexion in \Cref{R<->neutral}.(2) is necessary. Take an algebraically closed field $K$ of characteristic different from $2$, let $C_{4}$ be a cyclic group of order $4$, and $\chi$ a faithful character of $C_{4}$. Consider the $2$-dimensional representation $V = \chi \oplus \chi^{2}$ of $C_{4}$.
	
	The element of order $2$ of $C_{4}$ acts as a pseudoreflexion on $V$, so $V/C_{2}$ is smooth by the Chevalley--Shephard--Todd theorem and the singularity $(V/C_{4},[0]) = ((V/C_{2})/C_{2}, [0])$ is equivalent to the singularity $\left( (\rho \oplus \rho)/C_{2}, [0] \right)$, where $\rho$ is a faithful character of $C_{2}$. The representation $V$ of $C_{4}$ is neutral by \Cref{easy-cyclic}, but $(V/C_{4},[0]) \sim \left( (\rho \oplus \rho)/C_{2}, [0] \right)$ is not and $R$-singularity by \cite[Theorem 4]{giulio-singularities}.
\end{example}

\begin{corollary}
Let $G \subseteq \GL_{n}(K)$ be a finite discrete subgroup with $\cha k \nmid |G|$, which is generated by pseudoreflexions. Then the corresponding representation is neutral.
\end{corollary}

\begin{proof}
Set $V \eqdef \AA^{n}$. By the Chevalley--Shephard--Todd theorem, the quotient $V/G$ is smooth; hence $(V/G, [0])$ is an $R$-singularity, so the result follows from the theorem.
\end{proof}

\section{Blended eigenspaces}\label{blended}
In the rest of the paper we will discuss neutral representations of finite diagonalizable groups. The field $K$ will be algebraically closed, and $G$ will always be a finite diagonalizable group scheme over $K$. We will that $G$ is \emph{cyclic} if it is isomorphic to $\mmu_{n}$ for some natural number $n$.

We denote by $\widehat{G} \eqdef \hom_{K}(G, \gm)$ the group of characters of $G$. We will use additive notation for $\widehat{G}$.

Fix a representation $\rho_{V}\colon  G \xarr{\rho_{V}} \GL_{K}(V)$ of $G$ over $K$; we will usually simply say ``$V$ is a representation of $G$''. Then $V = \oplus_{\chi\in \widehat{G}}V_{\chi}$ has an eigenspace decomposition, where $G$ acts on $V_{\chi}$ by multiplication through $\chi$.

We denote by $\aut_{V} G$ the subgroup of the automorphism group $\aut G$ consisting of automorphisms $\phi\colon G \simeq G$ with the property that the composite representation $\rho_{V}\circ \phi\colon G \arr \GL_{K}(V)$ is isomorphic to $V$.  An automorphism $\phi$ of $G$ is in $\aut_{V}G$ if  and only if $\dim_{K}V_{\chi} = \dim_{K}V_{\chi\circ \phi}$ for every $\chi\in \widehat{G}$.

Denote by $\Omega_{V}$ the set of orbits for the natural action of $\aut_{V}G$ on $\widehat{G}$; clearly, if two characters $\chi$ and $\chi'$ are in the same orbit, we have $\dim V_{\chi} = \dim V_{\chi'}$. If $\omega \in \Omega_{V}$, we denote by $V_{\omega} \subseteq V$ the direct sum $\oplus_{\chi \in \omega}V_{\chi}$. Two characters $\chi$ and $\chi'$ are \emph{V-equivalent} if they are in the same orbit.
 
\begin{definition}
If $\omega \in \Omega_{V}$, we denote by $V_{\omega} \subseteq V$ the direct sum $\oplus_{\chi \in \omega}V_{\chi}$. The subrepresentations $V_{\omega} \subseteq V$ are called \emph{blended eigenspaces}.

A \emph{blended subrepresentation} of $V$ is a direct sum of blended eigenspaces.

The \emph{blended decomposition} of $V$ is the decomposition 
   \[
   V = \bigoplus_{\omega \in \Omega_{V}}V_{\omega}
   \]
\end{definition}

It is easy to see that a subspace $W \subseteq V$ is a blended subrepresentation if and only if it is preserved by every $\phi$-equivariant automorphism of $V$, where $\phi \in \aut_{V}G$.

The idea of the blended decomposition is to define the finest decomposition of $V$ which is guaranteed to induce a decomposition of every $V$-form. 

%

\begin{proposition}\label{prop:blend}
Let $\cV\to\cG$ be a $V$-form over $L$. The blended decomposition of $V$ induces a decomposition $\cV = \oplus_{\omega \in \Omega_{V}}\cV_{\omega}$. 
\end{proposition}

\begin{proof}
	Let $A$ be the band of $\cG$, and $A'$ a diagonalizable group scheme over $L$ such that $A_{\bar{L}}\simeq G$. Since $\underaut(G)$ is a finite étale group scheme and $\underisom_{L}(A,A')$ is an $\underaut_{L}(G)$-torsor, there exists a finite Galois extension $L'/L$ such that $A_{L'}\simeq A'_{L'}$ is diagonalizable. By \cite[Proposition 2.2.1.6]{lieblich}, $\cV_{L'}$ splits as a direct sum of eigensheaves corresponding to the eigenspaces of $V$. The statement follows by Galois descent.
\end{proof}

\section{Periods of abelian gerbes}\label{periods}

Let $A$ be an abelian, finite group scheme over a field $k$ and $\alpha\in\H^{i}(k,A)$ an fppf cohomology class for $i\ge 1$. The \emph{period} of $\alpha$ is its order in $\H^{i}(k,A)$, while the index of $\alpha$ is the greatest common divisor of the degrees of its splitting fields. The period of $\alpha$ divides the order of $A$.

It is well-known that, for étale (Galois) cohomology classes, the period divides the index and they have the same prime factors. This is still true for non-reduced group schemes and fppf cohomology.

\begin{lemma}\label{lemma:periodindex}
	The period of $\alpha$ divides its index, and they have the same prime factors.
\end{lemma}

\begin{proof}
	If $h/k$ is a splitting field, then $0=\operatorname{cor}_{h/k}\alpha_{h}=[h:k]\alpha\in\H^{2}(k, A)$, where $\operatorname{cor}\colon \H^{2}(h, A_{h}) \arr \H^{2}(k, A)$ is the corestriction operator. Hence the period divides the index. 
	
	Let $p$ be a prime which does not divide the period; we want to show that there exists a splitting field of degree prime with $p$. Let $h/k$ be a splitting field which is a normal extension, and let $h_{0}\subset h$ be the separable closure of $k$ in $h$. Let $G\subset\aut(h/k)=\operatorname{Gal}(h_{0}/k)$ be a $p$-Sylow subgroup, write $k'=h_{0}^{G}$ if $p=\cha k$ and $k'=h^{G}$ otherwise. In any case, $[h:k']$ is a power of $p$ and $[k':k]$ is prime with $p$.
	
	Since $h$ is a splitting field, the period of $\alpha_{k'}$ divides $[h:k']$; this implies that $\alpha_{k'}$ is trivial, since we are assuming that $p$ does not divide the period of $\alpha$. Hence, $k'/k$ is a splitting field for $\alpha$ of degree prime with $p$.
\end{proof}

A finite abelian gerbe $\cG$ over $k$ with band $A$ has an associated cohomology class $[\cG]\in\H^{2}(k,A)$, so we can talk about the period and index of $\cG$.

\section{The main theorems}\label{main-theorems}

Let us fix a finite diagonalizable group scheme $G$ over an algebraically closed field $K$, and a finite-dimensional representation $G \arr \GL_{K}(V)$. We want to give a criterion for $V$ to be neutral.

\begin{definition}
	A prime $p$ is \emph{critical} if there exists a $V$-form $\cV \to \cG$ such that $p$ divides the period of $\cG$.
\end{definition}

Thus, a representation is neutral if and only it has no critical primes.

Notice that every critical prime must divide $|G|$. 


If $p$ is a prime, denote by $G_{p}$ the largest subgroup scheme of $G$ whose order is a power of $p$. The character group $\widehat{G}_{p}$ is the $p$-primary part $\widehat{G}\{p\}$ of the character group $\widehat{G}$, and the restriction homomorphism $\widehat{G} \arr \widehat{G}_{p}$ induced by the embedding $G_{p} \subseteq G$ is the projection coming from the decomposition $\widehat{G} = \prod_{p}\widehat{G}\{p\}$. We will call $G_{p}$ the $p$-Sylow subgroup of $G$.

Let $\cV\to \cG$ be a $V$-form over a field $L$. By extending $K$, we may assume that $L \subseteq K$. Furthermore, we will fix an equivalence $(\alpha, \phi)$ of $(\cB_{K}G, [V/G])$ with $(\cG_{K}, \cV)$. This will allow us to associate with a line bundle $\cL$ on $\cG$, corresponding to a morphism $\cG \arr \cB_{L}\gm$, a character $G \arr \gm$, obtained by composing $\alpha\colon \cB_{K}G \simeq \cG_{K}$ with the morphism $\cG_{K} \arr\cB_{K}\gm$ coming from the pullback $\cL_{K}$ of $\cL$ to $\cG_{K}$. It is important to keep in mind that the character associated with $\cL$ depends on $\alpha$.

\begin{definition}
We say that the character $\chi\in \widehat{G}$ \emph{descends to $\cG$} if it comes from a line bundle on $\cG$.
\end{definition}

\begin{example}\label{example-descent}
Let $\omega\in \Omega_{V}$, and set $d \eqdef \dim V_{\chi}$ for $\chi \in \omega$. Then the character $\chi_{\omega} \eqdef d\sum_{\chi \in \omega}\chi$ descends to $\cG$, as it corresponds to the line bundle $\det \cV_{\omega}$ on $\cG$.
\end{example}


Here is our general criterion.

\begin{theorem}\label{generalcriterion}
Let $\cV\to \cG$ be a $V$-form over a field $L$, and $p$ a divisor of\/ $|G|$. Assume that there exists a finite extension $E/L$ of degree prime to $p$, such that the restrictions in $\widehat{G}_{p}$ of the characters of $G$ that descend to $\cG_{E}$ generate $\widehat{G}_{p}$.

Then $p$ does not divide the index of $\cG$.
\end{theorem}

In order to apply the theorem we need criteria to check when a character descends, after passing to an extension of degree prime to $p$.

\begin{theorem}\label{splitting}
Let $\cV\to \cG$ be a $V$-form over a field $L$, $p$ a divisor of\/ $|G|$, and $\omega$ in $\Omega_{V}$.

\begin{enumerate1}

\item If $p$ does not divide $|\omega|$, there exists a finite separable extension $E/L$ of degree prime to $p$, such that one of the characters in $\omega$ descends to $\cG_{E}$.

\item If $p > |\omega|$, there exists a finite Galois extension $E/L$ of degree prime to $p$, such that all the characters in $\omega$ descend to $\cG_{E}$.

\end{enumerate1}
\end{theorem}

\begin{corollary}
Let $p$ be a prime divisor of\/ $|G|$ with $p > \dim_{K}V$. If\/ $V$ is faithful when restricted to $G_{p}$, then $p$ is not critical.
\end{corollary}

\begin{corollary}
If $V$ is faithful, and $|G|$ is prime with $(\dim_{K}V)!$, then $V$ is neutral.
\end{corollary}

Together with \Cref{R<->neutral} and \cite[Lemma 6.22]{giulio-angelo-arithmetic}, this gives an alternative proof of \cite[Theorem 6.19]{giulio-angelo-arithmetic}.

The problem with applying \Cref{splitting} is that in case $p \leq |\omega|$, we don't have a way of knowing which of the characters in $\omega$ descends. The characters in $\widehat{G}$ whose image in $\widehat{G}_{p}/p\widehat{G}_{p}$ is zero do not contribute. Thus, if $\chi_{1}$, \dots,~$\chi_{r}$ are characters whose restriction to $\widehat{G}_{p}$ is not in $p\widehat{G}_{p}$, whether their images in $\widehat{G}_{p}$ generate only depends on the lines that their images in $\widehat{G}_{p}/p\widehat{G}_{p}$ generate (here we are thinking of $\widehat{G}_{p}/p\widehat{G}_{p}$ as a vector space on $\FF_{p}$). Denote by $P$ the projective space of lines in $\widehat{G}_{p}/p\widehat{G}_{p}$; if the action of $\aut_{V}G$ on $P$ is trivial, then whether the images in $\widehat{G}_{p}$ of a sequence $\chi_{1}$, \dots,~$\chi_{r}$ of characters  generate only depends on the classes of the $\chi_{i}$ in $\Omega_{V}$. Hence we get the following.

\begin{corollary}
Assume the following two conditions.
\begin{enumerate1}

\item The action of $\aut_{V}G$ on the set of lines in $\widehat{G}_{p}/p\widehat{G}_{p}$ is trivial.

\item Let $S$ be the set of characters $\chi \in \widehat{G}$ such that $p$ does not divide $\dim V_{\chi}$, and such that either

\begin{enumeratea}

\item the restriction to $G_{p}$ of $\sum_{\chi' \equiv \chi} \chi'$ is primitive, or

\item the order of the orbit of $\chi$ is not divisible by $p$.

\end{enumeratea}

Then the restrictions of the elements of $S$ to $G_{p}$ generate $\widehat{G}_{p}$.
\end{enumerate1}

Then $p$ is not critical.
\end{corollary}

When $G_{p}$ is cyclic, the result has a particularly simple form, given that the first condition of the corollary above is automatically satisfied, and that to generate $G_{p}$ we only need a faithful character.

\begin{corollary}\label{cyclic-general}
Suppose that $G_{p}$ is cyclic. Assume that there exists a faithful character $\chi \in \widehat{G}$ such that $p$ does not divide $\dim_{K}V_{\chi}$, and either 
\begin{enumeratea}

\item the restriction of $\sum_{\chi' \equiv \chi} \chi'$ to $G_{p}$ is faithful, or

\item the order of the orbit of $\chi$ is not divisible by $p$.

\end{enumeratea}

Then $p$ is not critical.

\end{corollary}


From this we get the following simple criterion for $V$ to be neutral when $G$ is cyclic.

\begin{corollary}\label{easy-cyclic}
Assume that $G$ is cyclic. For each prime divisor $p$ of\/ $|G|$, denote by $H_{p} \subseteq G$ the only subgroup scheme of $G$ that is isomorphic to $\mmu_{p}$. If the difference $\dim_{K}V - \dim_{K}V^{H_{p}}$ is not divisible by $p$ for every prime divisor $p$ of\/ $|G|$, then $V$ is neutral.
\end{corollary}

\begin{proof}
The point here is that a character of $G_{p}$ is faithful if and only if its restriction to $H_{p}$ is nontrivial.
\end{proof}

\section{The proofs of the theorems}\label{proofs}

\begin{proof}[Proof of \Cref{generalcriterion}]

Since $E/L$ has degree prime with $p$, then $p$ divides the index of $\cG$ if and only if it divides the index of $\cG_{E}$. Hence, up to replacing $L$ with $E$, we may assume that the restrictions in $\widehat{G}_{p}$ of the characters of $G$ that descend to $\cG$ generate $\widehat{G}_{p}$.

If $\Gamma \arr \spec L$ is the band of $\cG$, we have a decomposition $\Gamma = \Gamma' \times \Gamma''$, where $\Gamma', \Gamma''$ are group schemes of multiplicative type, such that the order of $\Gamma'$ is a power of $p$, and the order of $\Gamma''$ is prime with $p$. This gives a decomposition $\cG = \cG' \times \cG''$, where $\Gamma',\Gamma'$ are the bands of $\cG', \cG''$ respectively. It follows from \Cref{lemma:periodindex} that $\cG''$ has a splitting field $F$ of degree prime to $p$. 

If we show that $\cG'_{F}$ is neutral, we can conclude that $F/L$ is a splitting field for $\cG$, and hence, that $p$ does not divide the index of $\cG$. Thus, we may base change to $F$, and assume that $\cG'' = \cB_{L}\Gamma''$. 

Denote by $V_{p}$ the restriction of $V$ to $G_{p}$; the morphism $\spec L \arr \cB_{L}\Gamma''$ corresponding to the trivial $\Gamma''$-torsor yields a representable morphism $\cG' \arr \cG$. The pullback $\cV'$ of $\cV$ is a $V_{p}$-form. Hence we can replace $G$ with $G_{p}$, and prove the following particular case of the theorem.

\begin{lemma}
Assume that $G$ is a $p$-group, and that the characters of $G$ that descend to $\cG$ generate $\widehat{G}$. Then $\cG$ is neutral.
\end{lemma}

\begin{proof}
Choose a sequence $\chi_{1}$, \dots,~$\chi_{r}$ of characters of $G$ that descend to $\cG$, that forms a minimal set of generators of $\widehat{G}$. Let the order of $\chi_{i}$ be $p^{d_{i}}$, so that $\widehat{G} = (\ZZ/p^{d_{1}}\ZZ) \oplus \dots \oplus (\ZZ/p^{d_{r}}\ZZ)$. Let $\cL_{1}$, \dots,~$\cL_{r}$ be a sequence of line bundles on $\cG$, such that the pullback of $\cL_{i}$ to $\cB_{K}G \simeq \cG_{E}$ corresponds to $\chi_{i}$. 

Let $\cG \arr \cB_{L}\gm^{r}$ be the morphism obtained from the sequence $\cL_{1}$, \dots,~$\cL_{r}$. If $\Gamma$ is the band of $\cG$, the resulting homomorphism of commutative group schemes $\Gamma \arr \gm^{r}$ is injective, and the class of $\cG$ in $\H^{2}(L, \Gamma)$ maps to the class of $\cB_{L}\gm^{r}$, which is $0$, in $H^{2}(L, \gm^{r})$.

The homomorphism $\Gamma \arr \gm^{r}$ gives an isomorphism of $\Gamma$ with $\mmu_{p^{d_{1}}}\times \dots \times \mmu_{p^{d_{r}}} \subseteq \gm^{r}$, and $\H^{2}(L, \mu_{p^{d_{i}}}) \to \H^{2}(L, \gm)$ is injective by Hilbert's theorem 90. It follows that $\H^{2}(L, \Gamma) \to \H^{2}(L, \gm^{r}) = \H^{2}(L, \gm)^{r}$ is injective as well, hence the class of $\cG$ in $\H^{2}(L, \Gamma)$ is $0$ and $\cG$ is neutral.

\end{proof}

This ends the proof of \Cref{generalcriterion}.
\end{proof}

\begin{proof}[Proof of \Cref{splitting}]
As $G$ is diagonalizable, it comes from a unique diagonalizable group scheme over the prime field $k$, which we call again $G$; similarly, $V$ comes from a unique representation of $G$ over $k$, which we denote again by $V$.

Let us define two set-valued functors $\Gamma$ and $\Delta$ on $\aff[L]$ as follows. Let $n$ be the order of $\omega$. Let $S$ be an affine scheme over $L$.
\begin{enumerate1}

\item $\Gamma(S)$ is the set of subbundles $\cW \subseteq \cV_{S}$, with the property that there exists an fppf cover $\{S_{\alpha} \arr S\}$, an equivalence of $\cV_{S_{\alpha}} \simeq [V/G]_{S_{\alpha}}$, sending $\cW_{S_{\alpha}} \subseteq \cV_{S_{\alpha}}$ into $([V_{\chi}/G])_{S_{\alpha}}$ for some $\chi \in \omega$ (this $\chi$ might depend on $\alpha$).

\item $\Delta(S)$ is the set of sequences $\cW_{1}$, \dots,~$\cW_{n}$ of subbundles $\cW_{i} \subseteq \cV_{S}$, with the property that there exists an fppf cover $\{S_{\alpha} \arr S\}$, an equivalence of $\cV_{S_{\alpha}} \simeq [V/G]_{S_{\alpha}}$, sending the $(\cW_{i})_{S_{\alpha}} \subseteq \cV_{S_{\alpha}}$ into the set $([V_{\chi}])_{S_{\alpha}}$ for $\chi \in \omega$, with some ordering.

\end{enumerate1}

It is easy to see that $\Gamma$ is a finite étale cover of $\spec L$ of degree $n$; while $\Delta$ is a finite Galois cover with group $\rS_{n}$, where $\rS_{n}$ acts on $\Delta$ by permuting the $\cW_{i}$ above. If we take $E$ to be an extension of $L$ such that $\spec E$ is a connected component of $\Gamma$ in case (1), and $\Delta$ in case (2), we see that in both cases $E$ has the desired property.
\end{proof}

\section{Geometric applications}\label{applications}

We will refer to \cite{giulio-angelo-arithmetic} for our general setup on fields of moduli and residual gerbes. Assume that we are given a locally finitely presented fppf stack $\cM \arr \aff$. If $\xi$ is a an object in $\cM(\overline{k})$, we can define its residual gerbe $\cG_{\xi}$, is a subcategory of $\cM$ which is an fppf gerbe over an algebraic extension $\ell$ of $k$. Let us assume that $\cG_{\xi}$ is a finite  over $\ell$; if $k$ is perfect, by \cite[Proposition~3.9]{giulio-angelo-valuative}, this is equivalent to asking that the automorphism group scheme $G \eqdef \underaut_{\overline{k}}\xi$ is a finite group scheme.

Assume that for some $n$ we are given a cartesian functor $F\colon \cG_{\xi} \arr \cB_{k}\GL_{n}$, which associates with every object $\eta$ of some $\cG_{\xi}(S)$ a vector bundle $F(\eta)$ of rank~$n$ over $S$; this gives a vector bundle $\cV$ over $\cG_{\xi}$. We obtain a finite-dimensional vector space $V \eqdef F(\xi)$ over $\overline{k}$. By functoriality, the group $G$ acts linearly on $V$. By definition $(\cG, \cV)$ is equivalent to $(\cB_{\overline{k}}G, V)$. 

So, if the representation $G \arr \GL_{\overline{k}}V$ is neutral, the residual gerbe $\cG_{\xi}$ is neutral; in other words, $\xi$ is defined over its field of moduli.

For example, assume that $\cM$ is the stack of proper algebraic spaces; in other words, for each affine scheme $S$ over $k$, $\cM(S)$ is the category of proper finitely presented algebraic spaces over $S$. If $X \arr \spec \overline{k}$ is a proper algebraic space with residual gerbe $\cG_{X} \arr \spec \ell$ and $\pi\colon Y \arr S$ is a twisted form of $X$ defined over an affine scheme $S$, the sheaf $\rR^{i}\pi_{*}\Omega_{Y/S}$ is locally free, and its formation commutes with base change. We get a vector bundle $\cV$ on $\cG_{X}$; it is equivalent to the vector space $\H^{i}(X, \Omega_{X/\overline{k}})$, with the natural action of $G \eqdef \underaut_{\overline{k}}X$. 

We can will also consider the category $\cM$ of pointed spaces; an object $(X, x_{0})$ is a proper finitely presented algebraic space over $S$, and $x_{0}\colon S \arr X$ is a section of $X \arr S$. If $(X, x_{0})$ is an object of $\cM(S)$, we can consider the normal bundle $N_{S}X$. This is defined as the dual vector bundle to the locally free sheaf $\rI_{S}X/(\rI_{S}X)^{2}$ on $S$, where $\rI_{S}X$ is the sheaf of ideals of the closed embedding $x_{0}\colon S \arr S$. The corresponding representation of $ \underaut_{\overline{k}}(X, x_{0})$ is its natural action on the tangent space $\rT_{X,x_{0}} \eqdef (\frm_{x_{0}}/\frm_{x_{0}}^{2})^{\vee}$.

We will only give two simple examples of applications of this principle; much more general, and more elaborate, statements are of course possible. We will limit ourselves to the case of objects with cyclic automorphism groups, in which we can apply \Cref{easy-cyclic}; already in this simple case the statements are much more general than those present in the literature.

\begin{theorem}\label{marked-cyclic}
Let $X$ be a proper algebraic variety over $\overline{k}$, and let $x_{0} \in X(\overline{k})$ a smooth point. Assume that 
   \[
   G \eqdef \underaut_{\overline{k}}(X, x_{0}) = \mmu_{n}
   \]
for some positive integer $n$. For each prime divisor $p$ of $n$ denote by $H_{p} \subseteq G$ the only subgroup scheme that is isomorphic to $\mmu_{p}$. Assume that for each prime divisor $p$ of $n$ we have that $\dim_{x_{0}}X - \dim_{x_{0}}X^{H_{p}}$ is not divisible by $p$. Then the pointed variety $(X,x_{0})$ is defined over its field of moduli.
\end{theorem}

\begin{proof}
Call $V$ the tangent space to $X$ in $x_{0}$; then $G$ acts linearly on $V$. Let $H \subseteq G$ be a subgroup scheme of $G$; since $H$ is linearly reductive, we have that $X^{H}$ is smooth at $x_{0}$, and the tangent space to $X^{H}$ in $x_{0}$ is $V^{H}$. Hence the statement follows from \Cref{easy-cyclic}.\end{proof}

\begin{theorem}\label{curve-cyclic}
Let $X$ be a smooth projective curve over $\overline{k}$ of genus at least $2$, such that $G \eqdef \aut_{\overline{k}}X$ is cyclic, of order prime to $\cha k$. For each prime divisor $p$ of $|G|$ we denote by $H_{p} \subseteq G$ the only subgroup of order $p$. Assume that for every such $p$ the difference between the genus of $X$ and the genus of $X/H_{p}$ is not divisible by $p$. Then $X$ is defined over its field of moduli.
\end{theorem}

\begin{proof}
It is enough to prove that the representation $H^{0}(X, \Omega_{X/\overline{k}})$ of $G$ is neutral. Since the genus of $X$ equals $\dim_{\overline{k}}H^{0}(X, \Omega_{X/\overline{k}})$, while the genus of $X/H_{p}$ equals $\dim_{\overline{k}}H^{0}(X/H_{p}, \Omega_{(X/H_{p})/\overline{k}}) = \dim_{\overline{k}}H^{0}(X, \Omega_{X/\overline{k}})^{H_{p}}$, the statement follows from \Cref{easy-cyclic}.
\end{proof}

\bibliographystyle{amsalpha}
\bibliography{main}

\end{document}